\documentclass[a4paper,11pt]{amsart}
\usepackage{amsfonts}
\usepackage{amssymb}
\usepackage[utf8]{inputenc}
\usepackage{amsmath}
\usepackage{pdflscape}
\usepackage{graphicx}
\usepackage[colorlinks=true, linkcolor=red, citecolor=blue]{hyperref}
\usepackage[]{epsfig}
\usepackage[]{pstricks}
\usepackage{tikz}

\makeatletter
\@namedef{subjclassname@2020}{\textup{2020} Mathematics Subject Classification}
\makeatother

\numberwithin{equation}{section}

\newtheorem{thm}{Theorem}[section]
\newtheorem{claim}{Claim}
\newtheorem{lemma}[thm]{Lemma}

\newtheorem{prop}[thm]{Proposition}

\newtheorem{definition}{Definition}[section]

\newcommand{\R}{\mathbb{R}}

\newcommand{\C}{\mathbb{C}}
\newcommand{\N}{\mathbb{N}}
\newcommand{\Z}{\mathbb{Z}}

\newcommand{\NN}{\mathcal{N}}

\numberwithin{equation}{section}

\begin{document}
\title[Rapid Stabilization for the KdV--KdV system]{Rapid Exponential Stabilization of a Boussinesq System of KdV--KdV Type}
\author[Capistrano--Filho]{Roberto de A. Capistrano--Filho}
\address{Departamento de Matem\'atica, Universidade Federal de Pernambuco (UFPE), Avenida Professor Luiz Freire S/N, 50740-545, Recife (PE), Brazil.}
\email{roberto.capistranofilho@ufpe.br}
\author[Cerpa]{Eduardo Cerpa*}
\address{Instituto de Ingeniería Matemática y Computacional, Facultad de Matemáticas,  Pontificia Universidad Católica de Chile, Avda. Vicuña Mackenna 4860, Macul, Santiago, Chile}
\email{eduardo.cerpa@mat.uc.cl}
\author[Gallego]{Fernando A. Gallego}
\address{Departamento de Matem\'atica, Universidad Nacional de Colombia (UNAL), Cra 27 No. 64-60, 170003, Manizales, Colombia}
\email{fagallegor@unal.edu.co}
\subjclass[2020]{Primary: 93D15, 35Q53 Secondary: 93B05}
\keywords{KdV--KdV system, Gramian-based method, stabilization, feedback control}
\thanks{*Corresponding author: eduardo.cerpa@mat.uc.cl}

\begin{abstract}
This paper studies the exponential stabilization of a Boussinesq system describing the two-way propagation of small amplitude gravity waves on the surface of an ideal fluid, the so-called \textit{Boussinesq system of the Korteweg--de Vries type}. We use a Gramian-based method introduced by Urquiza to design our feedback control. By means of spectral analysis and Fourier expansion, we show that the solutions of the linearized system decay uniformly to zero when the feedback control is applied. The decay rate can be chosen as large as we want. The main novelty of our work is that we can exponentially stabilize this system of two coupled equations using only one scalar input. 
\end{abstract}

\maketitle

\section{Introduction}

\subsection{Setting of the problem} 
Boussinesq introduced in \cite{boussinesq1} several nonlinear partial differential equations to explain certain physical observations concerning the water waves, e.g. the emergence and stability of solitons. Unfortunately, several systems derived by Boussinesq proved to be ill-posed, so that there was a need to propose other systems with better mathematical properties.  In that direction, the four parameter family of Boussinesq systems
\begin{equation}
\left\{
\begin{array}
[c]{l}%
\eta_{t}+v_{x}+(   \eta v)  _{x}+av_{xxx}-b\eta_{xxt}=0\text{,}\\
v_{t}+\eta_{x}+vv_{x}+c\eta_{xxx}-dv_{xxt}=0,
\end{array}
\right.  \label{int_29e}%
\end{equation}
was introduced by  Bona \textit{et al.}  in \cite{bona-chen-saut}  to describe the motion of small amplitude long waves on the surface of an ideal fluid under the gravity force and in situations where the motion is sensibly two-dimensional. In (\ref{int_29e}), $\eta$ is the elevation of the fluid surface from the equilibrium position and $v$ is the horizontal velocity in the flow. The parameters $a$, $b$, $c$, $d$ are required to fulfill the relations
\begin{equation}
a+b=\frac{1}{2}(   \theta^{2}-\frac{1}{3})  \text{, \ \ \ }%
c+d=\frac{1}{2}(   1-\theta^{2})  \geq 0,
 \label{int_30e}%
\end{equation}
where $\theta\in\left[  0,1\right] $ and thus $a+b+c+d=\frac{1}{3}$. As it has been proved in \cite{bona-chen-saut}, the initial value problem for the {\em linear system} associated with \eqref{int_29e} is well posed on $\mathbb R$ if and only if the parameters $a,b,c,d$ fall in one of the  following cases
\begin{eqnarray*}
(\textrm{C}1)&& b,d\ge 0,\ a\le 0,\ c\le 0;\\
(\textrm{C}2)&& b,d\ge 0, \ a=c>0. 
\end{eqnarray*}

The well-posedness of the system  \eqref{int_29e} on the line $(x\in \R$) was investigated in \cite{BCS2}.  Considering $(C_2)$ with $b=d=0$, then necessarily $a=c=1/6$. Using the scaling $x\to x/\sqrt{6}$, $t\to t/\sqrt{6}$ gives a coupled system of two Korteweg--de Vries (KdV) equations  equivalent to \eqref{int_29e} for which $a=c=1$, namely
\begin{equation}
\label{b1}
\begin{cases}
\eta_t + w_x+w_{xxx}+  (\eta w)_x= 0, & \text{in} \,\, (0,L)\times (0,+\infty) ,\\
w_t +\eta_x +\eta_{xxx} +ww_x=0,  & \text{in} \,\, (0,L)\times (0,+\infty), \\
\eta(x,0)= \eta_0(x), \quad w(x,0)=  w_0(x), & \text{in} \,\, (0,L),
\end{cases}
\end{equation}
which is the so-called \textit{Boussinesq system of KdV--KdV type}. 

The goal of this paper is to investigate the boundary stabilization for  the linear Boussinesq system of KdV--KdV type
\begin{equation}\label{bb1}
\begin{cases}
\eta_t + w_x+w_{xxx}= 0, & \text{in} \,\, (0,L)\times (0,+\infty),\\
w_t +\eta_x +\eta_{xxx} =0,  & \text{in} \,\, (0,L)\times (0,+\infty), \\
\eta(x,0)= \eta_0(x), \quad w(x,0)=  w_0(x), & \text{in} \,\, (0,L),
\end{cases}
\end{equation}
with boundary conditions
\begin{equation}\label{b1.3}
\begin{cases}
\eta(0,t)=0,\,\,\eta(L,t)=0,\,\,\eta_{x}(0,t)=f(t),& \text{in} \,\, (0,+\infty),\\
w(0,t)=0,\,\,w(L,t)=0,\,\,w_{x}(L,t)=0,& \text{in} \,\, (0,+\infty), 
\end{cases}
\end{equation}
where $f(t)$ is the boundary control. We are mainly concerned with the following problem.
\vglue 0.2 cm
\noindent
\textit{{\bf Stabilization Problem:}
Can one find a linear feedback control law 
\begin{align*}
f(t) = F[(\eta(\cdot, t),w(\cdot,t)], \quad t \in (0,\infty),
\end{align*}
such that the closed-loop system \eqref{bb1} with boundary condition \eqref{b1.3} is exponentially stable?}
\vglue 0.2 cm


\subsection{Previous results} Abstract methods have been developed to obtain the rapid stabilization of linear partial differential equations. Among them, we cite the works \cite{komornik1997,urquiza2005,Vest}  based on the Gramian approach. In this paper we are interested to apply this method to design the feedback control law for the system \eqref{bb1}-\eqref{b1.3}. The method presented here was successfully applied by Cerpa and Cr\'epeau  in \cite{cerpa2009} to study the rapid stabilization of the KdV equation. Although this method is typically used in a single equation, this approach has not yet widespread applied to coupled systems.

Stability properties of systems \eqref{b1} or \eqref{bb1}  on a bounded domain have been studied by several authors. The pioneering work, for the system under consideration in this work, is due to Rosier and Pazoto in \cite{pazoto2008}. They showed the asymptotic behavior for the solutions of the system \eqref{bb1} satisfying the boundary conditions
\begin{equation}
\left\{
\begin{array}
[c]{lll}%
w(0,t)  =w_{xx}(0,t)  =0,\text{ } &  & \text{on $(0,T)$,}\\
w_{x}(0,t)  =\alpha_{0}\eta_{x}(0,t),  &  & \text{on $(0,T)$,}\\
w(L,t)  =\alpha_{2}\eta(L,t),  &  & \text{on $(0,T)$,}\\
w_{x}(L,t)  =-\alpha_{1}\eta_{x}(L,t),  &  & \text{on $(0,T)$,}\\
w_{xx}(L,t)  =-\alpha_{2}\eta_{xx}(L,t),  &  &\text{on $(0,T)$.}%
\end{array}
\right.  \label{int_32e_crp_1}%
\end{equation}
In (\ref{int_32e_crp_1}), $\alpha_{0}$, $\alpha_{1}$ and $\alpha_{2}$ denote some
nonnegative real constants.  Under the above boundary conditions, they observed that the derivative of the energy associated to the system (\ref{b1}), satisfies
\[
\frac{dE}{dt}    =-\alpha_{2}\left\vert \eta(   L,t)  \right\vert
^{2}-\alpha_{1}\left\vert \eta_{x}(   L,t)  \right\vert ^{2}%
-\alpha_{0}\left\vert \eta_{x}(   0,t)  \right\vert ^{2},
\]
where%
\[
E(   t)  =\frac{1}{2}\int_{0}^{L}(   \eta^{2}+w^{2})
dx\text{.}%
\]
This indicates that the boundary conditions play the role of a damping mechanism, at least for the linearized system.  In \cite{pazoto2008} the authors provides the following result.

\vglue 0.2 cm
\noindent
\textit{{\bf Theorem A }(Pazoto and Rosier  \cite{pazoto2008}) 
Assume that $\alpha_{0} \geq 0, \alpha_{1}>0$, and that $\alpha_{2}=1$ Then there exist two constants $C_{0}, \mu_{0}>0$ such that for any $\left(\eta_{0}, w_{0}\right) \in L^2(0,L)\times L^2(0,L)$, the solution of \eqref{bb1} with boundary condition \eqref{int_32e_crp_1} satisfies
$$\|(\eta(t), w(t))\|_{L^2(0,L)\times L^2(0,L)} \leq C_{0} \mathrm{e}^{-\mu_{0} t}\left\|\left(\eta_{0}, w_{0}\right)\right\|_{L^2(0,L)\times L^2(0,L)}, \quad \forall t \geq 0.$$}

Recently, Capistrano--Filho and Gallego \cite{caga2018} investigated the system \eqref{bb1} with two controls in the boundary conditions 
\begin{equation}\label{b1a}
\begin{cases}
\eta(0,t)=0,\,\,\eta(L,t)=0,\,\,\eta_{x}(0,t)=f(t),& \text{in} \,\, (0,+\infty),\\
w(0,t)=0,\,\,w(L,t)=0,\,\,w_{x}(L,t)=g(t),& \text{in} \,\, (0,+\infty),
\end{cases}
\end{equation}
and deal with the local rapid exponential stabilization by using the backstepping method. They designed boundary feedback controls  
\begin{align*} f(t)=F_1(\eta(\cdot,t), \omega(\cdot, t))\quad \text{and} \quad g(t)=F_2(\eta(\cdot,t), \omega(\cdot, t)),  
\end{align*} 
that leads to the stabilization of the system. The authors proved that the solution of the closed-loop system decays exponentially to zero in the $L^2(0,L)$--norm and the decay rate can be tuned to be as large as desired.
\vglue 0.2 cm
\noindent
\textit{{\bf Theorem B }(Capistrano--Filho and Gallego \cite{caga2018})  Let $L\in(   0,+\infty)  \backslash\mathcal{N}$ where 
\begin{equation}\label{criticalrapid}
\mathcal{N}:=\left\{\frac{2 \pi}{\sqrt{3}} \sqrt{k^{2}+k l+l^{2}} ; \quad k, l \in \mathbb{N}^{*}\right\} .
\end{equation}  
For every $\lambda>0$, there exist a continuous linear feedback control law $$F:=(F_1,F_2):L^2(0,L)\times L^2(0,L)\to\mathbb{R}\times\mathbb{R},$$ and positive constant  $C>0$. Then, for every $(\eta_0,w_0)\in L^2(0,L)\times L^2(0,L)$,
the solution $(\eta,w)$ of \eqref{bb1} with boundary conditions \eqref{b1a}
belongs to space $C([0,T];(L^2(0,L)\times L^2(0,L)))$ and satisfies}
\begin{equation*}
\|(\eta(t),w(t))\|_{L^2(0,L)\times L^2(0,L)} \leq C e^{-\frac{\lambda}{2} t} \|(\eta_0,w_0)\|_{L^2(0,L)\times L^2(0,L)}, \quad  \forall t\geq0.
\end{equation*}

\vglue 0.2 cm

It is important to emphasize that our goal in this paper is to stabilize the system \eqref{bb1} using only one feedback control, improving thus the result in \cite{caga2018}. Concerning controllability, the paper \cite{capistrano2016} studied different configurations for the position of the control, in particular they proved the following.
\vglue 0.2 cm
\noindent
\textit{{\bf Theorem C }(Capistrano--Filho \textit{et al.} \cite{capistrano2016}) Let $T>0$ and $L \in(0,+\infty) \backslash \mathcal{N}$. Then there exists some $\delta>0$ such that for all states $\left(\eta_{0}, w_{0}\right),\left(\eta_{1}, w_{1}\right) \in\left[L^{2}(0, L)\right]^{2}$ one can find a control $f \in L^{2}(0, T)$ such that the solution $$(\eta, w) \in C\left([0, T],\left[L^{2}(0, L)\right]^{2}\right) \cap L^{2}\left(0, T,\left[H^{1}(0, L)\right]^{2}\right)$$ of \eqref{bb1}-\eqref{b1.3} satisfies
$\eta(T, x)=\eta_{1}(x)$, $w(T, x)=w_{1}(x)$,  $x \in(0, L)$.}

\vspace{0.1cm}

As in the case of the KdV equation \cite[Lemma 3.5]{rosier}, when $L\in\mathcal{N}$, the linear system \eqref{bb1}-\eqref{b1.3} is not controllable. To prove Theorem C, the authors used the classical duality approach based upon the Hilbert Uniqueness Method (H.U.M.)  due to J.-L. Lions \cite{lions}, which reduces the exact controllability of the system to some observability inequality to be proved for the adjoint system. Then, to establish the required observability inequality, was used the compactness-uniqueness argument due to J.-L. Lions \cite{lions1} and some multipliers, which reduces the analysis to study a spectral problem. The spectral problem is finally solved by using a method introduced by Rosier in \cite{rosier}, based  on the use of complex analysis, namely, the Paley-Wiener theorem.


\subsection{Main result and outline of the work} Our result deals with the \textit{stabilization problem} already mentioned showing the following theorem.

\begin{thm}\label{main1} Let  $L\in(   0,+\infty)  \backslash\mathcal{N}$ and $\omega > 0$. Then, there exist a continuous linear map $$F_{\omega}: H_1    \longrightarrow \mathbb{R} $$ and a positive constant $C$, such that for every $(\eta_0,w_0)\in H_1$, the solution $(\eta,w)$ of the closed-loop system \eqref{bb1}-\eqref{b1.3}, with $f(t)=F_w(\eta(t),w(t))$ satisfies
\begin{equation*}
\|(\eta(t),w(t))\|_{H_1} \leq Ce^{-2\omega t}\|(\eta_0,w_0)\|_{H_1}, \quad \forall t\geq0.
\end{equation*}
\end{thm}

Theorem \ref{main1} is showed using the result proved by Urquiza in \cite[Theorem 2.1]{urquiza2005}. To use it, first, we need a spectral analysis (see Section \ref{Sec3} for the definition of the space $H_1$) for the operator $A$  given by
\begin{gather*}
A(\eta,w)=(-w'-w''',-\eta'-\eta''')
\end{gather*}
and
\begin{gather*}
D(A)=\left\lbrace (\eta,w)\in [H^3(0,L)\cap H^1_0(0,L)]^2: \eta'(0)=w'(L)=0\right\rbrace.
\end{gather*}
It is important to emphasize that our control acts only on one equation and through a boundary condition, instead of four or two controls as in Theorem A and B, respectively. 

\vspace{0.1cm}

The content of this article is divided as follows. Section \ref{Sec2} is devoted to presenting the Urquiza's approach, which requires four hypotheses to be satisfied called (H1), (H2), (H3) and (H4). In Section \ref{Sec3} we deal with some preliminary results including the proof of (H1). We will note that (H2) is easily verified. Next, we prove the hypotheses (H3) and (H4) in Section \ref{Sec4}. Section \ref{Sec5} is dedicated to the construction of the feedback and to finish the proof of Theorem \ref{main1}.  Some final comments are provided in Section \ref{sec6}.


\section{ Urquiza's approach} \label{Sec2}
In this section, we present the Urquiza method \cite{urquiza2005} to prove rapid exponential stabilization of the following system 
\begin{equation}\label{bb1a}
\begin{cases}
\eta_t + w_x+w_{xxx}= 0, & \text{in} \,\, (0,L)\times (0,+\infty),\\
w_t +\eta_x +\eta_{xxx} =0,  & \text{in} \,\, (0,L)\times (0,+\infty), \\
\eta(x,0)= \eta_0(x), \quad w(x,0)=  w_0(x), & \text{in} \,\, (0,L),
\end{cases}
\end{equation}
with boundary conditions
\begin{equation}\label{b1.3a}
\begin{cases}
\eta(0,t)=0,\,\,\eta(L,t)=0,\,\,\eta_{x}(0,t)=f(t),& \text{in} \,\, (0,+\infty),\\
w(0,t)=0,\,\,w(L,t)=0,\,\,w_{x}(L,t)=0,& \text{in} \,\, (0,+\infty).
\end{cases}
\end{equation}

We will extensively use the space operator given by 
\begin{gather}\label{operatorA}
A(\eta,w)=(-w'-w''',-\eta'-\eta'''),
\end{gather}
with domain
\begin{gather}\label{operatorA1}
D(A)=\left\lbrace (\eta,w)\in [H^3(0,L)\cap H^1_0(0,L)]^2: \eta'(0)=w'(L)=0\right\rbrace
\end{gather}
where
\begin{gather*}
D(A)\subset X_0:=[L^2(0,L)]^2 .
\end{gather*}

\subsection{Gramian method} Let us first write our system in the abstract framework. Set $A$ the operator defined by \eqref{operatorA}-\eqref{operatorA1} and $B$ given by
\begin{equation}\label{operatorB}
\begin{array}{c c c l }
B: & \mathbb{R}  & \longrightarrow &  D(A^*)'  \\
 & s & \longmapsto & B s 
\end{array}
\end{equation}
where $s\in \R$ and $Bs$ is a functional given by 
\begin{equation}\label{operatorB1}
\begin{array}{c c c l }
Bs: &  D(A^*)  & \longrightarrow &  \R \\
 & (u,v)& \longmapsto &  Bs(u,v):= - sv_x(0). 
\end{array}
\end{equation}
Since $D(A)=D(A^*)$ are closed and dense in $X_0$, we have that
\begin{equation}\label{adjunt}
\begin{array}{c c c l }
B^*: & D(A)   & \longrightarrow &  \mathbb{R} \\
 & (u,v) & \longmapsto & B^* (u,v)  =-v_x(0).
\end{array}
\end{equation}

Note that system \eqref{bb1a}-\eqref{b1.3a} takes the abstract form 
\begin{equation}\label{abstr}
\begin{cases}
\dot{y}(x,t)=Ay(x,t)+Bv(t), & \text{in $[D(A^*)]',$}\\
y(x,0)=y^0(x).
\end{cases}
\end{equation}
Here $y^0=(\eta_0,w_0)\in X_0$ is the initial condition  and the control is $v(t)=-f(t)$.

In order to get the rapid exponential stabilization, we use the Urquiza approach \cite{urquiza2005}. Let us explain the method on the abstract control system \eqref{abstr} with state $y(t)$ in a Hilbert space $Y$ and control $s(t)$ in a Hilbert space $U$. Here, the initial condition $y^0  \in  Y$, $A$ is a skew-adjoint operator  in $Y$ whose domain is dense in $Y$ , and $B$ is an unbounded operator from $U$ to $Y$ . 

The method to prove rapid stabilization consists on building a feedback control using the following four hypothesis for the operators $A$ and $B$:
\begin{enumerate}
\item[(H1)] The skew-adjoint operator $A$ is an infinitesimal generator of a strongly continuous group in the state space $Y$.
\vspace{0.1cm}
\item[(H2)] The operator $B:U \rightarrow D(A^*)'$ is linear and continuous.
\vspace{0.1cm}
\item[(H3)] \textit{(Regularity property)} For every $T>0$ there exists $C_T>0$ such that 
\begin{equation*}
\int_0^T \|B^*e^{-tA^*}y\|^2_{U}dt\leq C_T\|y\|_{Y}^2, \quad \forall y \in D(A^*).
\end{equation*}
\item[(H4)] \textit{(Observability property)} There exist $T>0$ and $c_T>0$ such that
\begin{equation*}
\int_0^T \|B^*e^{-tA^*}y\|^2_{U}dt\geq c_T\|y\|_{Y}^2, \quad \forall y \in D(A^*).
\end{equation*}
\end{enumerate}
With these hypotheses in hand, the next result holds. Its proof mainly relies on general results about the algebraic Riccati equation associated with the linear quadratic regulator problem (see \cite{lasiecka1988}).

\begin{thm}(See \cite[Theorem 2.1]{urquiza2005})
\label{urquizamethod}
 Consider operators $A$ and $B$ under assumptions (H1)-(H4). For any $\omega>0$, we have
\begin{enumerate}
\item[(i)] The symmetric positive operator $\Lambda_{\omega}$ defined by 
\begin{align*}
(\Lambda_{\omega}x,z)_{Y}= \int_0^{\infty} \left(B^*e^{-\tau (A+\omega I)^*}x, B^*e^{-\tau (A+\omega I)^*}z\right)_{U}d\tau, \quad \forall x,z \in Y,
\end{align*}
is coercive and is an isomorphism on $ Y$.
\vspace{0.1cm}
\item[(ii)] Let $F_{\omega}:=-B^*\Lambda_{\omega}^{-1}$. The operator $A+BF_{\omega}$ with $D(A+BF_{\omega})=\Lambda_{\omega}(D(A^*))$ is the infinitesimal generator of a strongly continuous semigroup on $Y$.
\vspace{0.1cm}
\item[(iii)]The closed-loop system (system \eqref{abstr} with the feedback law $v=F_{\omega}(y)$) is exponentially stable with a decay equals to $2\omega$, that is, 
\begin{align*}
\exists C>0, \forall y \in Y, \quad \|e^{t(A+BF_{\omega})}y\|_{ Y}\leq C e^{-2\omega t}\|y\|_{ Y}.
\end{align*}
\end{enumerate}
\end{thm}

In order to apply this method, we have to verify the four hypotheses. This will be done in the next sections. 


\section{Preliminaries}\label{Sec3}
To apply Theorem \ref{urquizamethod} to our linear Boussinesq control system is necessary to check the hypotheses (H1)-(H4).  First, we will prove that operator $A$ defined by \eqref{operatorA} satisfies (H1). Note that by the definition of the operator $B$, see \eqref{operatorB1}, it is easy to see that (H2) also follows true.   Moreover, in this section we establish the asymptotic behavior of some eigenfunctions and give the definition of some useful spaces.

\subsection{Hypothesis (H1)}We first comment that in order to verify hypothesis (H1) it is enough to prove that $A$ is skew-adjoint in $X_0$.
\begin{prop}\label{prop1}
$A$ is a skew-adjoint $X_0$ and thus generates a group of isometries $(e^{tA})_{t\in \R}$ in $X_0$.
\end{prop}
\begin{proof}
First, it  is clear that $D(A)$ is dense in $X_0$. We have to prove that $A^*=-A$. Note that we have $-A\subset A^*$ (i.e. $(\theta , u)\in D(A^*)$ and  
$A^*(\theta , u) =-A(\theta ,u)$ for all $(\theta ,u)\in D(A)$). Indeed, 
for any $(\eta , v),(\theta , u)\in D(A)$, we have after some integration by parts that 
$$((\theta , u) ,A(\eta ,v))_{X_0}= -(A(\theta, u), (\eta , v))_{X_0}.$$

Let us prove now that $A^*\subset -A$. Pick any $(\theta , u ) \in D(A^*)$. Then,  we have for some constant 
$C>0$ 
\[
\left\vert ( (\theta , u ) , A(\eta , v))_{X_0}\right\vert \le C \Vert (\eta , v)\Vert _{X_0}\quad \forall (\eta , v) \in D(A), 
\] 
i.e.
\begin{equation}
\label{A2}
\left\vert \int _0^L  [ \theta (v_x+v_{xxx} ) + u (\eta _x + \eta _{xxx}) ]dx \right\vert 
  \le C \left(  \int_0^L [\eta ^2 + v^2 ]dx\right)^\frac{1}{2},
\qquad \forall (\eta , v)\in D(A). 
\end{equation}
Picking $v=0$ and $\eta \in C_c^\infty (0,L)$, we infer from \eqref{A2} that $u_x+u_{xxx}\in L^2(0,L)$, and hence that 
$u\in H^3(0,L)$. Similarly, we obtain that $\theta \in H^3(0,L)$. 
Integrating by parts in the left hand side  of \eqref{A2}, we obtain that 
\begin{eqnarray*}
&&\left\vert \theta (L) v_{xx}(L) -\theta (0)v_{xx}(0) + \theta _x (0)v_x(0) 
+ u(L) \eta _{xx}(L)  -u(0) \eta _{xx} (0) -u_x(L) \eta _x(L) \right\vert \\
&& \qquad \le C\left( \int_0^L [\eta ^2 + v^2] dx \right) ,\qquad \forall (\eta , v) \in D(A) . 
\end{eqnarray*}
It easily follows that 
\[
\theta (0)=\theta (L)=\theta _x(0)=u(0)=u(L)=u_x(L)=0,
\]
so that $(\theta , u)\in D(A)=D(-A)$. Thus $D(A^*)=D(-A)$ and $A^*=-A$, which ends the proof of this proposition.
\end{proof}

\subsection{Behavior of the traces}
As can be seen in \cite{capistrano2016}, operator $A$ has a compact resolvent and it can be diagonalized in an orthonormal basis, i.e., the spectrum $\sigma(A)$ of $A$ consists only of eigenvalues and the eigenfunctions form an orthonormal basis of $X_0$. Thus, due to the results presented in \cite{capistrano2016}, there exists an orthonormal basis $\left\lbrace (\theta^+_n,u_n^+)_{n\in \Z} \cup  (\theta^-_n,u_n^-)_{n\in \Z}\right\rbrace$ in $[L^2_{\C}(0,L)]^2$, endowed with the natural scalar product $$\left((\theta, u), (\varphi,\omega)\right)=\int_0^L\left(\theta(x)\overline{\varphi(x)}+u(x)\overline{\omega(x)}\right)dx,$$ composed of eigenfunctions of $A$ satisfying
\begin{align*}
A (\theta^+_n,u_n^+) &= i\lambda_n  (\theta^+_n,u_n^+)
\end{align*}
and
\begin{align*}
A (\theta^-_n,u_n^-) &= (-i\lambda_n ) (\theta^-_n,u_n^-),
\end{align*}
where the real numbers $\{\lambda_n\}_{n\in \Z}$ have the following asymptotic forms
\begin{align}\label{eee1}
\lambda_n = \begin{cases}
\left( \dfrac{\pi+12\pi(k_1+n)}{6L}\right)^3 + O(n), & \text{as $n \rightarrow +\infty$} \\
\\
-\left( \dfrac{7\pi+12\pi(k_2-n)}{6L}\right)^3 + O(n), & \text{as $n \rightarrow -\infty$} \\
\end{cases}
\end{align}
for some numbers $k_1,k_2 \in \Z.$ Next result provides the behavior of boundary traces associated with the orthonormal basis 
$\left\lbrace (\theta^+_n,u_n^+)_{n\in \Z} \cup  (\theta^-_n,u_n^-)_{n\in \Z}\right\rbrace$. The proof is given in Appendix \ref{appendix}.
\begin{prop}\label{behavor2}
There exist positive constants $C_1^{\pm}$ and $C_2^{\pm}$, such that  
\begin{equation}\label{behaveigenfunctions}
\lim_{|n|\rightarrow \infty} \frac{| \theta^{\pm}_{n,x}(L))|}{|n|}=C_1^{\pm} \qquad  \text{and} \qquad \lim_{|n|\rightarrow \infty} \frac{| u_{n,x}^{\pm}(0)|}{|n|}=C_2^{\pm}.
\end{equation}
\end{prop}

\subsection{Definition of ad-hoc spaces}  In order to verify later the hypothesis (H3) and (H4), we have to change the state space $X_0$. Since $$\left\lbrace (\theta^+_n,u_n^+)_{n\in \Z} \cup  (\theta^-_n,u_n^-)_{n\in \Z}\right\rbrace$$ is an orthonormal basis in $[L^2_{\C}(0,L)]^2$, we have that for any $(f,g) \in [L^2_{\C}(0,L)]^2$, there exists a unique sequence $\{c^{n,+}_{f,g}\}_{n\in \Z} \cup \{c^{n,-}_{f,g}\}_{n\in \Z}$, with $\sum_{n\in\Z}\left( |c^{n,+}_{f,g}|^2 +|c^{n,-}_{f,g}|^2\right)<\infty$, such that
$$
(f,g)=\sum_{n\in\Z} \left( c^{n,+}_{f,g}(\theta^+_n,u_n^+) + c^{n,-}_{f,g}(\theta^-_n,u_n^-) \right)$$
and $$ \|(f,g)\|_{[L^2_{\C}(0,L)]^2} =\left( \sum_{n\in\Z}|c^{n,+}_{f,g}+c^{n,-}_{f,g}|^2\right)^{1/2}.$$

Let $Z=span\{    (\theta^+_n,u_n^+) \cup  (\theta^-_n,u_n^-) \}$. For any $s \in \R$, consider the norm 
\begin{align*}
\left\|\sum_{n\in\Z} c^{n,+}  (\theta^+_n,u_n^+) + c^{n,-}  (\theta^-_n,u_n^-)\right\|_s:=  \left( \sum_{n\in\Z} (1+|\lambda_k|)^{\frac{2}{3}s}(|c^{n,+}+c^{n,-}|^2)\right)^{\frac12}.
\end{align*}
Let us now define the $H_s$ spaces, for $s \in \R$.
\begin{definition}\label{spacesH}
The spaces $H_s$ will be defined as the completion of $Z$ with respect of the norm $\|\cdot\|_s$.
In each space $H_s$, one has the orthonormal basis $$\{(1+|\lambda_n|)^{-\frac{s}{3}}\theta^+_n, (1+|\lambda_n|)^{-\frac{s}{3}} u_n^+ \}_{n\in\Z}\cup \{(1+|\lambda_n|)^{-\frac{s}{3}}\theta^-_n, (1+|\lambda_n|)^{-\frac{s}{3}} u_n^- \}_{n\in\Z}.$$
\end{definition}


\section{Proof of hypothesis (H3) and (H4)} \label{Sec4} In this section we are interested to prove the hypothesis (H3) and (H4). We start presenting some auxiliary results that will be used to prove the regularity condition and observability inequality, respectively.

\subsection{Auxiliary lemmas} To find the regularity needed and to prove the observability inequality we use the following classical Ingham inequality, see e.g., \cite{ingham1936} and \cite{komornik2005} for details.
\begin{lemma}\label{ingham}
Let $T>0$ and $\{\beta_n\}_{n\in \Z}\subset\R$ be a sequence of pairwise distinct real numbers such that 
\begin{align*}
\lim_{|n|\rightarrow \infty} (\beta_{n+1}-\beta_n)=+\infty.
\end{align*}
Then, the series  $$h(t)=\sum_{n\in\Z}\gamma_ne^{i\beta_n t} \text{ converges in }L^2(0,T), $$ for any sequence $\{\gamma_n\}_{n\in\Z}$ satisfying $\sum_{n\in\Z}\gamma_n^2<\infty$. Moreover, there exist two strictly positive constant $C_1$ and $C_2$ such that,
\begin{align*}
C_1\sum_{n\in\Z}\gamma_n^2 \leq \int_0^T |h(t)|^2dt \leq C_2 \sum_{n\in\Z}\gamma_n^2.
\end{align*}
\end{lemma}

Now, observe that following the spectral analysis for  the operator $A$ given in \cite[Appendix: Proof of Theorem 3.11]{capistrano2016}, we obtain that 
\begin{align}\label{uv}
\theta^{\pm}n(x):= \mp\dfrac{i}{\sqrt{2}}v_n(L-x) \quad \text{and} \quad
u_n^{\pm}(x)=\dfrac{1}{\sqrt{2}}v_n(x),
\end{align}
where $\{v_n\}_{n\in \Z}$ are the eigenvectors of the operator $\mathcal{B}$ defined as 
\begin{align*}
\mathcal{B}y=-y'''(L-x)-y'(L-x),
\end{align*}
with domain $D(\mathcal{B})=\left\lbrace y \in H^3(0,L)\cap H^1_0(0,L): y'(L)=0\right\rbrace$, which is closely related with the operator $A$. So, operator $\mathcal{B}$ has the following properties that can be seen in \cite[Appendix: Proof of Theorem 3.11]{capistrano2016}.
\begin{lemma}\label{lemma42}The operator $\mathcal{B}$  is self-adjoint in $L^2(0,L)$. Moreover, the following claims hold: 
\begin{enumerate}
\item[(i)] If $L \in (0,\infty)  \setminus \NN$, then $$\mathcal{B}^{-1}: L^2(0,L)\rightarrow H^{3}(0,L)$$ is well defined continuous operator. Here, $\NN$ is defined by \eqref{criticalrapid};
\vspace{0.2cm}
\item[(ii)] There is an orthonormal basis $\{v_n\}_{n\in \Z}$ in $L^2(0,L)$ composed of eigenvectors of $\mathcal{B}$: $v_n \in D(\mathcal{B})$ and $\mathcal{B}v_n = \lambda_n v_n$ for all $n \in \N$ for some $\lambda_n \in \R$.
\end{enumerate}
\end{lemma}

Finally, the net result is a direct consequence of the spectral analysis for  the operator $A$ and ensures the well-posedness for the homogeneous system associated to the system \eqref{bb1a}-\eqref{b1.3a}.
\begin{lemma}\label{teo}
For any $(\eta_0,w_0)=\sum_{n\in \Z}\left( z^{n,+}_{0} (\theta^+_n,u_n^+) + z^{n,-}_{0} (\theta^-_n,u_n^-)\right) \in H_s$, there exists a unique solution of 
\begin{equation*}
\begin{cases}
(\eta_t,w_t)=A(\eta,w), \\
(\eta(0),w(0))=(\eta_0,w_0),
\end{cases}
\end{equation*}
belonging of $C(\R,H_s)$ and given by 
\begin{equation*}
(\eta(x,t), w(x,t))=\sum_{n\in\Z}e^{i\lambda_n t} \left( z^{n,+}_{0} (\theta^+_n,u_n^+) + z^{n,-}_{0} (\theta^-_n,u_n^-)\right).
\end{equation*}
Additionally, as $\{\lambda_k\}_{k\in\mathbb{Z}} \subset \R$, we have $$\|(\eta(t),w(t))\|_{s}=\|(\eta_0,w_0)\|_s.$$
\end{lemma}

\subsection{Proof of (H3)}  Let us take $$(\eta_0,w_0)=\sum_{n\in \Z}\left( z^{n,+}_{0} (\theta^+_n,u_n^+) + z^{n,-}_{0} (\theta^-_n,u_n^-)\right) \in H_1,$$ for $\beta_n=\lambda_n$
and
\begin{align*}
&\delta_n = z_0^{n,+}u^+_{n,x}(0)+ z_0^{n,-}u^-_{n,x}(0), 
\end{align*}
where $\theta^{\pm}_{n,x}$ and $u_{n,x}^{\pm}$ are defined in \eqref{uv}. Initially, note that \eqref{eee1}  implies that sequences $\{\beta_n\}_{n\in \N}$ satisfies 
\begin{align*}
\lim_{|n|\rightarrow \infty} (\beta_{n+1}-\beta_n)=+\infty.
\end{align*}
Thanks to \eqref{uv}, we deduce the following quantities
\begin{align}\label{4.2}
\begin{cases}
\theta^+_{n,x}(L):= -\dfrac{i}{\sqrt{2}}v_{n,x}(0), \quad  \overline{\theta^-_{n,x}(L)}:= -\dfrac{i}{\sqrt{2}}\overline{v_{n,x}(0)},  \\
u_{n,x}^+(0)=\dfrac{1}{\sqrt{2}}v_{n,x}(0), \quad \overline{u_{n,x}^-(0)}=\dfrac{1}{\sqrt{2}} \overline{v_{n,x}(0)}.
\end{cases}
\end{align}
Thus, it follows that 
\begin{align*}
 \sum_{n\in\Z}\delta_n^2&=\sum_{n\in\Z} \left|z_0^{n,+} u^+_{n,x}(0)+ z_0^{n,-} u^-_{n,x}(0)\right|^2  \\
 &= \sum_{n\in\Z} |z_0^{n,+} |^2 |u^+_{n,x}(0)|^2 + 2Re \left(  z_0^{n,+}  \overline{z_0^{n,-} } u^+_{n,x}(0) \overline{u^-_{n,x}(0)} \right)+ |z_0^{n,-}|^2 |u^-_{n,x}(0)|^2 \\
 &= \sum_{n\in\Z} \frac{ |v_{n,x}(0)|^2}{2} \left( |z_0^{n,+} |^2-2Re \left(  z_0^{n,+}  \overline{z_0^{n,-} } \right)+|z_0^{n,-}|^2 \right) \\
 &= \frac{1}{2}\sum_{n\in\Z}  |v_{n,x}(0)|^2  |z_0^{n,+} +z_0^{n,-}|^2.
\end{align*}
Hence, 
\begin{align*}
\sum_{n\in\Z}\delta_n^2 = \frac{1}{2}\sum_{n\in\Z}  \frac{|v_{n,x}(0)|^2}{(1+|\lambda_n|)^{\frac{2}{3}}} \left[  (1+|\lambda_n|)^{\frac{2}{3}}|z_0^{n,+} +z_0^{n,-}|^2\right].
\end{align*}

Using the asymptotic behavior \eqref{eee1}, there exists a positive constant $C_1$ such that
\begin{align*}
 \frac{ |v_{n,x}(0)|^2}{(1+|\lambda_k|)^{\frac{2}{3}}}\leq C_1,  \quad \forall n \in \Z.
\end{align*}
Thus, 
\begin{align}\label{otro2}
\sum_{n\in\Z}\delta_n^2 &\leq C_1\sum_{n\in\Z} (1+|\lambda_k|)^{\frac{2}{3}}\left( |z_0^{n,+}+ z_0^{n,-}|^2  \right) = C_1\|(\eta_0,w_0)\|_{1}^2.
\end{align}
On the other hand, note that by Lemma \ref{teo}, we have
\begin{equation*}
(\eta(x,t), w(x,t))=\sum_{n\in\Z}e^{i\lambda_k t} \left( z^{n,+}_{0} (\theta^+_n,u_n^+) + z^{n,-}_{0} (\theta^-_n,u_n^-)\right).
\end{equation*}
Thus, its implies that  
$$w_x(0,t)=\sum_{n\in\Z}e^{i\lambda_n t} \left( z_0^{n,+}u^+_{n,x}(0)+ z_0^{n,-}u^-_{n,x}(0)\right)=\sum_{n\in\Z}\delta_ne^{i\lambda_n t}.$$ From Lemma \ref{ingham} and  relations \eqref{uv}
and \eqref{otro2}, we obtain that 
\begin{equation}\label{otro4}
C_2 \sum_{n\in\Z} \left|z_0^{n,+} u^+_{n,x}(0)+ z_0^{n,-} u^-_{n,x}(0)\right|^2  \leq \int_0^T |w_x(0)|^2dt \leq C_1\|(\eta_0,w_0)\|_{1}^2,
 \end{equation}
for some positive constants $C_1$ and $C_2$.  Thanks to \eqref{otro4}, condition $(H3)$ holds in $H_1$-norm.  

\subsection{Proof of (H4)}  We can see using \eqref{4.2} and \eqref{otro4} that
\begin{align}\label{observability}
\begin{split}
 \int_0^T |w_x(0)|^2dt 
 \geq &\ C_3 \sum_{n\in\Z}  |v_{n,x}(0)|^2  |z_0^{n,+} +z_0^{n,-}|^2 \\
 = &\ C_3 \sum_{n\in\Z}  \frac{|v_{n,x}(0)|^2}{(1+|\lambda_n|)^{\frac{2}{3}}} \left[  (1+|\lambda_n|)^{\frac{2}{3}}|z_0^{n,+} +z_0^{n,-}|^2\right],
 \end{split}
\end{align}
holds for some $C_3>0$.
Observe that we can estimate the right hand side of \eqref{observability} in terms of any $H_s$--norm for $s\geq 1$. To finalize the proof of the hypothesis (H4), for  $H_1$--norm, we can not lose any coefficient $z_0^{n,\pm}$. Thus, we claim the following. 
\begin{claim}\label{claim2}
$v_{n,x}(0) \neq 0$ for all  $n \in \Z.$
\end{claim}
Indeed, suppose by contradiction  that there exists $n_0 \in \Z$ such that $v_{n_0,x}(0) = 0$. This implies that 
\begin{equation}\label{thethan}
(\theta^+_{n_0,x}(L), u^+_{n_0,x}(0)) = 0 \quad \text{and}\quad (\theta^-_{n_0,x}(L), u^-_{n_0,x}(0)) = 0.
\end{equation}
In particular, considering $u(x)= \theta_{n_0}^+(x) + u_{n_0}^+(x)$, there exist $\lambda_{n_0} \in \C$ such that 
\begin{equation*}
\begin{cases}
u'''+u' +\lambda_{n_0}u =0, \\
u(0)=u(L)=u'(0)=u'(L)=0.
\end{cases}
\end{equation*}
From \cite[Lemma 3.5]{rosier}, it follows that $L \in \mathcal{N}$, which is a contradiction, and Claim \ref{claim2} holds.

Lastly, due to the asymptotic behavior \eqref{eee1} and Claim \ref{claim2}, there exits $C_4>0$ such that $$\frac{|v_{n,x}(0)|^2}{(1+|\lambda_n|)^{\frac{2}{3}}} \geq C_4,$$ for all $n\in \N$, Thus, follows by \eqref{thethan} that 
\begin{equation}\label{H4}
\int_0^T |w_x(0)|^2 dt   \geq C_4\|(\eta_0,w_0)\|_{1}^2  
\end{equation}
Therefore, relation (H4) is satisfied.


\section{Rapid stabilization: Control design}\label{Sec5}

In this section we design the feedback law using the Urquiza's approach to show the rapid exponential stabilization for solutions of the system \eqref{bb1a}-\eqref{b1.3a}. Recall that this system takes an abstract form \eqref{abstr}  and the operators $A$ and $B$ are given by \eqref{operatorA}-\eqref{operatorA1} and \eqref{operatorB}-\eqref{operatorB1}, respectively. With this in hand, we are in position to prove our main result.

\subsection{Proof of Theorem \ref{main1}} 
For any $(p_0,q_0)$, $(r_0,s_0)$ in $H_1$ and $\omega>0$, consider the bilinear form defined by 
\begin{equation}\label{form_1}
a_{\omega}((p_0,q_0),(r_0,s_0)):=\int_0^{\infty} e^{-2\omega \tau} q_x(0,\tau)s_x(0,\tau)d\tau.
\end{equation}
Here $(p,q)$ and $(r,s)$ are solutions of 
\begin{equation*}
\begin{cases}
p_{\tau}+q_x+q_{xxx}=0, \\
q_{\tau}+p_x+p_{xxx}=0, \\
p(0,\tau)=p(L,\tau)=p_x(0,\tau)=0, \\
q(0,\tau)=q(L,\tau)=q_x(L,\tau)=0, \\
p(x,0)=p_0(x), \quad  q(x,0)=q_0(x)
\end{cases}
\end{equation*}
and 
\begin{equation*}
\begin{cases}
r_{\tau}+s_x+s_{xxx}=0, \\
s_{\tau}+r_x+r_{xxx}=0, \\
r(0,\tau)=r(L,\tau)=r_x(0,\tau)=0, \\
s(0,\tau)=s(L,\tau)=s_x(L,\tau)=0, \\
r(x,0)=r_0(x), \quad  s(x,0)=s_0(x),
\end{cases}
\end{equation*}
respectively.  Finally, consider the following operator  $\Lambda_{\omega}: H_1 \longrightarrow H_{-1}$  satisfying the relation
\begin{equation}\label{form1_1}
\left\langle \Lambda_{\omega} (p_0,q_0),(r_0,s_0)\right\rangle_{H_{-1},H_1}=a_{\omega}((p_0,q_0),(r_0,s_0)), \quad \forall (p_0,q_0),(r_0,s_0) \in H_1 \text{ and } \forall\omega>0. 
\end{equation}
Note that, from \eqref{adjunt} we have that 
\begin{align*}
\left\langle \Lambda_{\omega} (p_0,q_0),(r_0,s_0)\right\rangle_{H_{-1},H_1} & = \int_0^{\infty} e^{-2\omega \tau} q_x(0,\tau) s_x(0,\tau) d\tau, \\
&= \int_0^{\infty} e^{-2\omega \tau} B^*(p(x, \tau), q(x, \tau )) B^*(r(x, \tau), s(x, \tau ))d\tau.
\end{align*}
Thanks to the Theorem \ref{urquizamethod}, the operator $\Lambda_{\omega}$ is coercive and an isomorphism. On the other hand, set the functional 
\begin{equation*}
\begin{array}{l l l l}
F_{\omega}: & H_1 &   \longrightarrow & \R \\
            & (z_1,z_2)   &  \longmapsto          &F_{\omega}((z_1,z_2)):=q_0'(0), 
\end{array}
\end{equation*}
where $(p_0,q_0)$ is the solution of the following Lax-Milgram problem 
\begin{equation}\label{laxmilgram_1}
a_{\omega}((p_0,q_0),(r_0,s_0))=\left\langle (z_1,z_2),(r_0,s_0)\right\rangle_{H_{-1},H_1}, \quad \forall (r_0,s_0) \in H_1.
\end{equation}
From \eqref{form1_1}, we deduce that $(z_1,z_2)= \Lambda_{\omega} (p_0,q_0)$ in $H_{-1}$. Moreover, observe that 
\begin{align*}
F_{\omega}(z_1,z_2)=q_0'(0)=- B^* (p_0,q_0)= - B^* \Lambda_{\omega}^{-1} (z_1,z_2), \quad \forall (z_1,z_2) \in H_{1}. 
\end{align*}

Thus, we are in the hypothesis of Theorem \ref{urquizamethod}, which one can be applied and guarantees the rapid exponential stabilization to the solutions of the system \eqref{bb1a}-\eqref{b1.3a}. It means that  for any $\omega > 0$,  there exists a continuous linear feedback control
\begin{align*}
f(t)=F_\omega(\eta(t),w(t))
\end{align*}
with $F_{\omega} =-B^*\Lambda_{\omega}^{-1}$ where  $\Lambda_{\omega}$ is given by  \eqref{form_1}-\eqref{form1_1} and a positive constant $C$, such that for every initial conditions $(\eta_0,w_0) \in H_1$, the solution $(\eta,w)$ of the closed-loop system \eqref{bb1a}-\eqref{b1.3a}, satisfies
\begin{equation*}
\|(\eta(t),w(t))\|_{H_1} \leq Ce^{-2\omega t}\|(\eta_0,w_0)\|_{H_1},
\end{equation*}
with a decay equals to $2\omega$.
\qed

\section{Further comments}\label{sec6}
We have applied the Gramian approach to build some boundary feedback law to prove the rapid stabilization for a coupled KdV--KdV type system. Considering one control acting on the Neumann boundary condition at the right-hand side of the interval where the system evolves we are able to prove that the closed-loop system is locally exponentially stable with a decay rate that can be chosen to be as large as we want. Below we present some final remarks.
\begin{itemize}
\item[$\bullet$]Theorem A guarantees the stabilization of the KdV-KdV system with four controls and Theorem B ensures the rapid stabilization with two controls. However, Theorem \ref{main1} gives us a best result for the linear system, that is, we are able to make the solutions of the linear system go to zero with only one control acting at the boundary. It is important to point out here that the drawback is that we are not able to treat the nonlinear case. This is due to the lack of any Kato smoothing effect, as in the case of a single KdV \cite{cerpa2009}, which leaves the rapid stabilization for the full system \eqref{b1} with boundary conditions \eqref{b1.3} completely open to study.
\item[$\bullet$] Note that we can also prove that  there exists a continuous linear feedback control $$g(t)=F_\omega(\eta(t),w(t))$$ such that the closed-loop system \eqref{bb1} with boundary conditions 
\begin{equation*}
\begin{cases}
\eta(0,t)=0,\,\,\eta(L,t)=0,\,\,\eta_{x}(0,t)=0,& \text{in} \,\, (0,+\infty),\\
w(0,t)=0,\,\,w(L,t)=0,\,\,w_{x}(L,t)=g(t),& \text{in} \,\, (0,+\infty),\\
\end{cases}
\end{equation*}
satisfies
\begin{equation*}
\|(\eta(t),w(t))\|_{H_1} \leq Ce^{-2\omega t}\|(\eta_0,w_0)\|_{H_1},
\end{equation*}
with a decay equals to $2\omega$. To prove this consider  the operator $B$ given by
\begin{equation*}
\begin{array}{c c c l }
B: & \mathbb{R}  & \longrightarrow &  D(A^*)'  \\
 & s & \longmapsto & B s := L_s
\end{array}
\end{equation*}
where $s \in \R$ and $L_s$ is a functional given by 
\begin{equation*}
\begin{array}{c c c l }
L_s: &  D(A^*)  & \longrightarrow &  \R \\
 & (u,v)& \longmapsto &  L_s(u,v):= su_x(L).
\end{array}
\end{equation*}
With these information in hand and the following observability inequality 
$$
\int_0^T |\eta_x(L)|^2  dt 
 \geq C\|(\eta_0,w_0)\|_{1}^2, \quad C>0,
 $$  the result  follows using the same idea as done in the proof of  Theorem \ref{main1}.
\end{itemize}   

\appendix

\section{Proof of Proposition \ref{behavor2}}\label{appendix}
Following the ideas of \cite[Appendix A. Proof of Theorem 3.11]{capistrano2016}, we observe that $v_n$ takes the form
\begin{equation}\label{new8}
v_n(x)=\sum_{j=1}^3 a_j\left[ e^{r_jx}-ie^{r_j(L-x)}\right],
\end{equation}
with $a_j=a_j(n) \in  \C$, for $j=1,2,3$, where
\begin{equation}\label{new}
\begin{split}
\sum_{j=1}^{3} a_{j}\left(e^{r_{j} L}-i\right)=0, \\ 
\sum_{j=1}^{3} a_{j}\left(1-i e^{r_{j} L}\right)=0, \\ 
\sum_{j=1}^{3} r_{j} a_{j}\left(e^{r_{j} L}+i\right)=0
\end{split}
\end{equation}
and $r_j$, $j=1,2,3$, are pairwise distinct such that
\begin{equation}\label{r1}
r_1=r_1(n) \sim -i\lambda^{1/3}_n, \quad r_2=r_2(n) \sim -ip\lambda^{1/3}_n, \quad r_3=r_3(n) \sim -ip^2\lambda^{1/3}_n,
\end{equation}
for $p=e ^{i\frac{2\pi}{3}}$. Note that the equations in \eqref{new} imply that 
\begin{equation*}
\sum_{j=1}^{3} a_{j}=\sum_{j=1}^{3} a_{j} e^{r_{j} L}=0,
\end{equation*}
that is, 
$$
a_{3} =-a_{1}-a_{2}
$$ 
and 
$$a_{1}\left(e^{r_{1} L}-e^{r_{3} L}\right)+a_{2}\left(e^{r_{2} L}-e^{r_{3} L}\right) =0.$$
Moreover, if we assume $\lambda_n\to\infty$, we have
\begin{align*}
& r_1 = -i\lambda^{1/3}_n + O(\lambda^{-1/3}_n)\sim -i\lambda^{1/3}_n \\
& r_2=-ip\lambda^{1/3}_n +  O(\lambda^{-1/3}_n) \sim  \left( \frac{\sqrt{3}}{2}+\frac{i}{2}\right)\lambda^{1/3}_n
\end{align*}
and
\begin{align*}
&r_3= -ip^2\lambda^{1/3}_n  + O(\lambda^{-1/3}_n)  \sim \left( -\frac{\sqrt{3}}{2}+\frac{i}{2}\right)\lambda^{1/3}_n.
\end{align*}
Additionally, if $\lambda_n\to-\infty$, we have
\begin{align*}
& r_1 = -i\lambda^{1/3}_n + O(\lambda^{-1/3}_n)\sim i | \lambda^{1/3}_n|, \\
& r_2=-ip\lambda^{1/3}_n +  O(\lambda^{-1/3}_n) \sim  -\left( \frac{\sqrt{3}}{2}+\frac{i}{2}\right)| \lambda^{1/3}_n|
\end{align*}
and
\begin{align*}
&r_3= -ip^2\lambda^{1/3}_n  + O(\lambda^{-1/3}_n)  \sim \left( \frac{\sqrt{3}}{2}-\frac{i}{2}\right)| \lambda^{1/3}_n|.
\end{align*}
Note that the previous relations implies
\begin{equation}\label{new5}
|e^{r_1L}| \rightarrow 1, \quad |e^{r_2L}| \rightarrow +\infty, \quad |e^{r_3L}| \rightarrow 0, \quad \text{as $n\rightarrow \infty$}
\end{equation}
and
\begin{equation}\label{new5'}
|e^{r_1L}| \rightarrow 1, \quad |e^{r_2L}| \rightarrow 0, \quad |e^{r_3L}| \rightarrow +\infty, \quad \text{as $n\rightarrow -\infty$.}
\end{equation}
The convergences \eqref{new5} and \eqref{new5'} ensures that, $$\lambda_n \rightarrow \pm\infty,\quad \text{ as }n\rightarrow\pm\infty.$$

With these relation in hand, the following claim can be verified.

\begin{claim}
\textit{The behaviors \eqref{behaveigenfunctions} hold whenever there exist positive constant $C_1$ and $C_2$ such that}
\begin{equation}\label{new1}
\lim_{|n|\rightarrow \infty} \frac{| v'_n(0)|}{|n|}=C_1 \quad \text{and} \quad  \lim_{|n|\rightarrow \infty} \frac{| v'_n(L)|}{|n|}=C_2.
\end{equation}
\end{claim}
In fact, to obtain the limit \eqref{new1}, we have to analyze the asymptotic behavior of the $a_j(n)$ terms. First, note that $\|v_n\|_{L^2(0,L)}=1$, since $\{v_n\}_{n\in \Z}$ is an orthonormal basis thanks to the Lemma \ref{lemma42}, thus
\begin{align}\label{new2}
1 &= \int_0^L \left | \sum_{j=1}^3 a_j\left[ e^{r_jx}-ie^{r_j(L-x)}\right] \right| ^2dx = \int_0^L  \left( \sum_{j=1}^3 A_j^2(x)  +2 \sum_{i,j=1, i\neq j}^3  A_i(x)A_j(x) \right) dx,
\end{align}
where $A_j(x)= a_j\left[ e^{r_jx}-ie^{r_j(L-x)}\right] $. Then, 
\begin{align*}
\int_0^L A_j^2(x)dx &= \int_0^L a_j^2 \left( e^{2r_jx} - 2ie^{r_j x}e^{r_j(L-x)}-e^{2r_j(L-x)} \right) dx = a_j^2 \left[ \frac{e^{2r_j x}}{2r_j} -2ie^{r_j L}x+  \frac{e^{2r_j (L-x)}}{2r_j}\right]_0^L 
\end{align*}
hence 
\begin{equation}\label{new3}
\int_0^L A_j^2(x)dx = - 2i a_j^2e^{r_j L}L.
\end{equation}
On the other hand, 
\begin{align*}
\int_0^L A_i(x) A_j(x) dx &= \int_0^L a_ia_j \left( e^{r_i x} - ie^{r_i(L-x)} \right)\left( e^{r_jx} - ie^{r_j(L-x)} \right) dx  \\
&= a_ia_j \left[ \frac{e^{(r_i+r_j)x}}{r_i+r_j} - \frac{ie^{(r_i-r_j)x}e^{r_jL}}{r_i-r_j}+ \frac{ie^{-(r_i-r_j)x}e^{r_iL}}{r_i-r_j} +
\frac{e^{(r_i+r_j)(L-x)}}{r_i+r_j} \right]_0^L,
\end{align*}
therefore
\begin{equation}\label{new4}
\int_0^L A_i(x) A_j(x) dx = 2ia_ia_j\frac{e^{r_jL}-e^{r_iL}}{r_i-r_j}, \quad \forall i \neq j.  
\end{equation}
Putting together \eqref{new3} and \eqref{new4} in \eqref{new2}, we get
\begin{multline}\label{newa}
-2iL\left(a_1^2 e^{r_1L} + a_2^2 e^{r_2L}+a_3^2 e^{r_3L}\right) \\+ 4i\left( a_1a_2\frac{e^{r_2L}-e^{r_1L}}{r_1-r_2} +a_1a_3\frac{e^{r_3L}-e^{r_1L}}{r_1-r_3} +a_2a_3\frac{e^{r_3L}-e^{r_2L}}{r_2-r_3} \right)=1.
\end{multline}
Thanks to the second and first relation in \eqref{new}, respectively, we have
\begin{equation*}
a_2=- \Gamma a_1, \quad  a_3= -a_1(1-\Gamma),
\end{equation*}
where $\Gamma = \left( e ^{r_1L}-e ^{r_3L}\right) \left( e ^{r_2L}-e ^{r_3L}\right)^{-1}$.  Now, using \eqref{newa}, we obtain
\begin{multline}\label{eqnew1}
-2iLa_1^2\left(e^{r_1L} + \Gamma^2 e^{r_2L}+(1-\Gamma)^2 e^{r_3L}\right) \\
 + 4ia_1^2\left(-\Gamma\frac{e^{r_2L}-e^{r_1L}}{r_1-r_2} -(1-\Gamma)\frac{e^{r_3L}-e^{r_1L}}{r_1-r_3}+ \Gamma(1-\Gamma)\frac{e^{r_3L}-e^{r_2L}}{r_2-r_3}  \right) =1.
\end{multline}
Moreover, note that 
\begin{align*}
1-\Gamma &= \frac{e^{r_2L}-e^{r_1L}}{e^{r_2L}-e^{r_3L}}
\end{align*}
and
\begin{align*}
\Gamma(1-\Gamma) &= \frac{(e^{r_1L}-e^{r_3L})(e^{r_2L}-e^{r_1L})}{(e^{r_2L}-e^{r_3L})^2}.
\end{align*}
Therefore, we have 
\begin{align*}
\Gamma(1-\Gamma)\frac{e^{r_3L}-e^{r_2L}}{r_2-r_3} &=- \frac{(e^{r_1L}-e^{r_3L})(e^{r_2L}-e^{r_1L})}{(r_2-r_3)(e^{r_2L}-e^{r_3L})},
\end{align*}
\begin{align*}
\Gamma\frac{e^{r_2L}-e^{r_1L}}{r_1-r_2}&=  \frac{(e^{r_1L}-e^{r_3L})(e^{r_2L}-e^{r_1L})}{(r_1-r_2)(e^{r_2L}-e^{r_3L})}
\end{align*}
and
\begin{align*}
(1-\Gamma)\frac{e^{r_3L}-e^{r_1L}}{r_1-r_3}&= - \frac{(e^{r_1L}-e^{r_3L})(e^{r_2L}-e^{r_1L})}{(r_1-r_3)(e^{r_2L}-e^{r_3L})}.
\end{align*}
Using the previous equalities in \eqref{eqnew1}, follows that
\begin{multline*}
-2iLa_1^2\left(e^{r_1L} + \Gamma^2 e^{r_2L}+(1-\Gamma)^2 e^{r_3L}\right) \\
 +   4ia_1^2  \frac{(e^{r_1L}-e^{r_3L})(e^{r_2L}-e^{r_1L})}{(e^{r_2L}-e^{r_3L})}  \left( \frac{ 1}{r_1-r_3}- \frac{1}{r_2-r_3} -\frac{1}{r_1-r_2}  \right)  =1
\end{multline*}
or, equivalently, 
\begin{align*}
-2ia_1^2\left[L\left(e^{r_1L} + \Gamma^2 e^{r_2L}+(1-\Gamma)^2 e^{r_3L}\right) 
 - 2\Gamma  (e^{r_2L}-e^{r_1L}) \phi (r_1,r_2,r_3) \right]=1,
\end{align*}
where $$\phi (r_1,r_2,r_3) =\frac{ 1}{r_1-r_3}- \frac{1}{r_2-r_3} -\frac{1}{r_1-r_2}.$$
Noting that
\begin{equation*}
|a_1(n)|^2= \frac{1}{2\left|L\left( e^{r_1L} + \Gamma^2 e^{r_2L}+(1-\Gamma)^2 e^{r_3L}\right) -    2\Gamma   (e^{r_2L}-e^{r_1L}) \phi (r_1,r_2,r_3)  \right|},
\end{equation*}
follows that
\begin{multline}\label{new6}
\begin{aligned}
\frac{1}{2 \left( L|e^{r_1L}| +L| \Gamma^2 e^{r_2L}|+L|(1-\Gamma)^2 e^{r_3L} |+ 2|\Gamma | |e^{r_2L}-e^{r_1L}|| \phi (r_1,r_2,r_3)|\right)} \\
 \leq |a_1(n)|^2 \leq \frac{1}{2\left| L|e^{r_1L}| - L|\Gamma^2 e^{r_2L}| -L|(1-\Gamma)^2 e^{r_3L}|- 2|\Gamma | |e^{r_2L}-e^{r_1L}|| \phi (r_1,r_2,r_3)|\ \right|}.
 \end{aligned}
\end{multline}
Additionally, from \eqref{r1}, we have 
\begin{align*}
& r_1-r_2 = -i(1-p)\lambda^{1/3}_n + O(\lambda^{-1/3}_n), \\
& r_2-r_3=-ip(1-p)\lambda^{1/3}_n +  O(\lambda^{-1/3}_n)
\end{align*}
and
\begin{align*}
&r_1-r_3= -i(1-p^2)\lambda^{1/3}_n  + O(\lambda^{-1/3}_n),
\end{align*}
which allow us to conclude that $\phi (r_1,r_2,r_3)\rightarrow 0$, as $|n|\rightarrow\infty$. 

Observe that
\begin{align*}
|\Gamma| = \left| \frac{(e^{r_1L}-e^{r_3L})}{e^{r_2L}(1-e^{r_3L}e^{-r_2L})} \right| \leq  \frac{ |e^{r_1L}|+|e^{r_3L}|}{|e^{r_2L}||1-e^{r_3L}e^{-r_2L}|}
\end{align*}
\begin{align*}
|\Gamma^2 e^{r_2L}| = \left| \frac{(e^{r_1L}-e^{r_3L})^2}{e^{r_2L}(1-e^{r_3L}e^{-r_2L})^2} \right| \leq  \frac{( |e^{r_1L}|+|e^{r_3L}|)^2}{|e^{r_2L}||1-e^{r_3L}e^{-r_2L}|^2}.
\end{align*}
and
\begin{align*}
|(1-\Gamma)^2e^{r_3L}| = \left| \left(1-\frac{e^{r_1L}-e^{r_3L}}{e^{r_2L}-e^{r_3L}}\right)^2e^{r_3L}\right|\leq \frac{( |e^{r_2L}|+|e^{r_1L}|)^2}{|e^{r_3L}||e^{r_2L}e^{-r_3L}-1|^2}.
\end{align*}
Due to the asymptotic behavior of $\{\lambda_n\}_{n\in\mathbb{N}}$ \eqref{eee1} and \eqref{new5}, we get
\begin{equation}\label{exp1}
|\Gamma| \rightarrow 0, \quad |\Gamma^2 e^{r_2L}| \rightarrow 0 \quad \text{and} \quad |(1-\Gamma)^2 e^{r_3L}|\rightarrow 0, \quad  \text{as $n \rightarrow \infty$.}
\end{equation}
Note that $|\Gamma e^{r_2L}| \rightarrow 1$, which implies that $$|\Gamma | |e^{r_2L}-e^{r_1L}| | \phi (r_1,r_2,r_3)|  \rightarrow 0 \quad  \text{ as } n \rightarrow \infty.$$ On the other hand, thanks to \eqref{new5'} follows that
\begin{equation}\label{exp2}
|\Gamma| \rightarrow 1, \quad  |\Gamma||e^{r_2L}-e^{r_1L}| \rightarrow 1, \quad |\Gamma^2 e^{r_2L}| \rightarrow 0 \quad  \text{and} \quad |(1-\Gamma)^2e^{r_3L}| \rightarrow 0, \quad  \text{as $n \rightarrow -\infty$}.
\end{equation}
Thus, using $\eqref{exp1}$, $\eqref{exp2}$ and passing to the limit in \eqref{new6}, we deduce that
\begin{equation}\label{new7}
\lim_{|n|\rightarrow \infty} |a_1(n)| = \frac{1}{\sqrt{2L}}.
\end{equation}

Let us prove \eqref{new1}. From \eqref{new8} we have that  $$v_n'(x) =\sum_{j=1}^3a_jr_j \left[ e^{r_jx}+ie^{r_j(L-x)}\right],$$ then
$$
v_n'(0) = \sum_{j=1}^3a_jr_j \left[ 1+ie^{r_jL}\right]
$$
and
$$v_n'(L) = \sum_{j=1}^3a_jr_j \left[ e^{r_jL}+i \right].$$
Thanks to \eqref{new}, it follows that
\begin{align*}
v_n'(0) &= a_1r_1 \left( 1+ie^{r_1L}\right) +a_2r_2 \left( 1+ie^{r_2L}\right)-(a_1+a_2)r_3 \left( 1+ie^{r_3L}\right)\\
&= a_1r_1 \left( 1+ie^{r_1L}\right) +a_2(r_2-r_3)+ a_2r_2i(e^{r_2L}-e^{r_3L}) +a_2ie^{r_3L}(r_2-r_3)-a_1r_3 \left( 1+ie^{r_3L}\right) \\
&= a_1r_1 \left( 1+ie^{r_1L}\right) +a_2(r_2-r_3)- a_1r_2i(e^{r_1L}-e^{r_3L}) +a_2ie^{r_3L}(r_2-r_3)-a_1r_3 \left( 1+ie^{r_3L}\right) \\
&= a_1 \left[  r_1 \left( 1+ie^{r_1L}\right)  -r_3-r_2ie^{r_1L}\right]     +a_2(r_2-r_3)(1+ie^{r_3L}) + a_1ie^{r_3L}(r_2 -r_3) \\
&= a_1 \left[  r_1 \left( 1+ie^{r_1L}\right)  -r_3-r_2ie^{r_1L}\right]   -\frac{a_1(e^{r_1L}-e^{r_3L})}{(e^{r_2L}-e^{r_3L})}(r_2-r_3)(1+ie^{r_3L}) + a_1ie^{r_3L}(r_2 -r_3) 
\end{align*}
We analyze the case when $n\rightarrow +\infty$, the case when $n\rightarrow -\infty$ can be shown  analogously. Noting that
$$
\frac{e^{r_1L}-e^{r_3L}}{e^{r_2L}-e^{r_3L}}=O\left( e^{-\frac{\sqrt{3}}{2}\lambda^{1/3}L}\right),
$$
$$r_2-r_3=O\left(\lambda^{1/3}\right)$$
and 
$$ 1+ie^{r_3L}=O\left(1+ e^{-\frac{\sqrt{3}}{2}\lambda^{1/3}L}\right),$$
we have that
\begin{align*}
\frac{a_1(e^{r_1L}-e^{r_3L})}{(e^{r_2L}-e^{r_3L})}(r_2-r_3)(1+ie^{r_3L}) + a_1ie^{r_3L}(r_2 -r_3)=O\left(e^{-\frac{\sqrt{3}}{2}\lambda^{1/3}L}\right).
\end{align*}
 Note that $(1-p^4)= (1-p) $, due the fact that $p=e ^{i\frac{2\pi}{3}}$. Thus,
\begin{align*}
v_n'(0) & = a_1\left(   (r_1-r_3)  +ie^{r_1L} (r_1-r_2)  + O\left( \lambda_n^{1/3}e^{-\frac{\sqrt{3}}{2}\lambda_n^{1/3} L}\right)\right) \\
&= a_1\left(   (-i\lambda_n^{1/3})(1-p^2)  +ie^{r_1L} (-i\lambda_n^{1/3})(1-p)  + O\left( 1\right)\right) \\
&= a_1\left(   (-i\lambda_n^{1/3})(1-p^2)  +ie^{r_1L} (-i\lambda_n^{1/3})(1-p^4)  + O\left( 1\right)\right) \\
&= a_1 (-i\lambda_n^{1/3})(1-p^2)\left[1  +ie^{r_1L}(1+p^2)  + O\left( 1\right)\right] \\
&= a_1 (-i\lambda_n^{1/3})(1-p^2)\left(1  + ie^{r_1L}p^2  + O\left( 1\right)\right).
\end{align*}
Since $ e^{r_1L} \sim e^{-i\lambda_n^{1/3}L}\sim e^{-i\pi/6} \sim ip^2$, 
there exists $K_1^+ \in \C\setminus\{0\}$, such that
\begin{align*}
v_n'(0)  \sim K_1^+ a_1(n) \lambda_n^{1/3}.
\end{align*}
Similarly, there exists $K_2^+ \in \C\setminus\{0\}$, such that
\begin{align*}
v_n'(L)  \sim K_2^+ a_1(n) \lambda_n^{1/3}.
\end{align*}
Analogously,  when $n\rightarrow -\infty$, we get $$v_n'(0)  \sim K_1^{-} a_1(n) \lambda_n^{1/3}$$ and $$v_n'(L)  \sim K_2^{-} a_1(n) \lambda_n^{1/3},$$ for some complex constants nonzero $K_1^{-}$ and $K_2^{-}$, respectively. 
Finally, \eqref{eee1} and \eqref{new7} ensure that \eqref{new1} follows and, consequently, Proposition \ref{behavor2} is achieved.

\subsection*{Acknowledgments:}  R. de A. Capistrano--Filho was supported by CNPq 408181/2018-4, CAPES-PRINT 88881.311964/2018-01, CAPES-MATHAMSUD 88881.520205/2020-01, MATHAMSUD 21- MATH-03 and Propesqi (UFPE). E. Cerpa was supported by ANID Millennium Science Initiative Program NCN19-161 and Basal Project FB0008. F. A. Gallego was supported by MATHAMSUD 21-MATH-03.


\begin{thebibliography}{99}
\bibitem {bona-chen-saut}Bona J. J., Chen M. and Saut J.-C., \emph{Boussinesq equations and other systems for small-amplitude long waves in nonlinear
dispersive media. I. Derivation and linear theory}, J. Nonlinear Science, 12
(2002), 283--318.

\bibitem {BCS2} Bona, J. J., Chen, M., and  Saut, J.-C.,
\emph{Boussinesq equations and other systems for small-amplitude long waves in
nonlinear dispersive media. II. The nonlinear theory}. Nonlinearity 17 (2004) 925--952.

\bibitem {boussinesq1} Boussinesq, J. V., \emph{Th\'{e}orie g\'{e}n\'{e}rale
des mouvements qui sont propag\'{e}s dans un canal rectangulaire horizontal}.
C. R. Acad. Sci. Paris 72 (1871), 755--759.



\bibitem{cerpa2009}
E. Cerpa and E. Crépeau, \textit{Rapid exponential stabilization for a linear Korteweg-de Vries equation.} Discrete Contin. Dyn. Syst. Ser. B, 11(3) (2009),  655--668.






\bibitem{caga2018}
R. A. Capistrano--Filho and F.A. Gallego, \textit{Asymptotic behavior of Boussinesq system of KdV--KdV type}. J. Differential Equations. (265) (2018), 2341--2374.

\bibitem{capistrano2016}
R. A. Capistrano--Filho, A. F. Pazoto and L. Rosier, \textit{Control of Boussinesq system of KdV--KdV type on a bounded interval}. ESAIM Control Optimization and Calculus Variations 25 58 (2019), 1--55.

\bibitem{lasiecka1988}
F. Flandoli, I. Lasiecka and R. Triggiani,\textit{ Algebraic Riccati equations with non-smoothing observation arising in hyperbolic and Euler-Bernoulli boundary control problems}. Ann. Mat. Pura Appl., 153 (1988), 307--382.

\bibitem{ingham1936}
A.E. Ingham, \textit{Some trigonometrical inequalities with applications to the theory of series}. Math. Z., 41 (1936), 367--379.

\bibitem{lions}
J.-L. Lions, Contrôlabilité exacte, perturbations et stabilisation de systèmes distribués, Tome1, Recherches en Mathématiques Appliquées (Research in Applied Mathematics), vol.8, Masson, Paris, (1988).

\bibitem{lions1}
J.-L. Lions, Exact controllability, \textit{Stabilization and perturbations for distributed systems}. SIAM Rev., 30 (1) (1988), 1--68.

%


\bibitem{komornik1997}
V. Komornik, \textit{Rapid boundary stabilization of linear distributed systems}. SIAM journal on control and optimization, 35(5) (1997), 1591--1613.

\bibitem{komornik2005}
V. Komornik and P. Loreti, Fourier Series in Control Theory, Springer Monographs in Mathematics, Springer-Verlag, New York, (2005).

\bibitem{pazoto2008}
A. F. Pazoto and L.  Rosier, \textit{Stabilization of a Boussinesq system of KdV–KdV type}. Systems and Control Letters, 57 (8) (2008), 595--601. 

\bibitem{rosier} 
L.  Rosier, \textit{Exact  boundary controllability for the Korteweg-de Vries  equation on a bounded domain}. ESAIM Control Optim. Cal. Var, 2 (1997), 33--55.

%

\bibitem{urquiza2005}
J.M. Urquiza, \textit{Rapid exponential feedback stabilization with unbounded control operators}. SIAM journal on control and optimization, 43(6) (2005), 2233--2244.

\bibitem{Vest}A. Vest. \textit{Rapid stabilization in a semigroup framework}. SIAM J. Control Optim., 51 (5) (2013), 4169--4188.



\end{thebibliography}
\end{document}